\documentclass[11pt]{article}   
\usepackage[greek,english]{babel}  
\usepackage{graphicx}
\usepackage[pdftex,bookmarks]{hyperref}
\usepackage{amsmath}
 \usepackage{amssymb}
\usepackage{amsmath,amsthm}
\usepackage{verbatim}
\usepackage{sidecap}
\usepackage{tikz}
\newtheorem{theorem}{Theorem}[section]

\newtheorem{lemma}[theorem]{Lemma}

\newtheorem{proposition}[theorem]{Proposition}

\theoremstyle{definition}
\newtheorem{remark}[theorem]{Remark}

\def\ep{\varepsilon}

\titlepage 
\title{\Large{\bf{Periodic solutions of a perturbed Kepler problem in the plane: from existence to stability}}}
\author{Alberto Boscaggin and Rafael Ortega}
\date{}

\begin{document}
\maketitle

\begin{abstract}
The existence of elliptic periodic solutions of a perturbed Kepler problem is proved. The equations are in the plane and the perturbation
depends periodically on time. The proof is based on a local description of the symplectic group in two degrees of freedom.
\end{abstract}

\vspace{0.4cm}
\noindent
\footnotesize{\textbf{AMS-Subject Classification}}. {\footnotesize 37J25, 37J45.} 

\normalsize

\section{Introduction}

Perturbations of the Kepler problem appear naturally in Celestial Mechanics. These equations are relevant for applications
but they also have an intrinsic mathematical intererest. In particular the existence and stability of periodic solutions have
been discussed by a large number of authors. After Poincar\'e these questions are usually treated via the averaging method.
We refer to the papers \cite{Mos70,Yan08} for results on the autonomous case and for useful lists of references.
We will be interested in periodic time dependent perturbations, a class of equations already considered by Fatou in \cite{Fa28}.
More recently Ambrosetti and Coti Zelati treated in \cite{AmbCot89} a class of periodic perturbations with symmetries and presented
the averaging method in a variational framework.

We are going to consider the perturbed Kepler problem in the plane
\begin{equation}\label{pk}
\ddot x = -\frac{x}{\vert  x \vert^3} + \ep \, \nabla_x U(t,x), \qquad x \in \mathbb{R}^2 \setminus \{0\},
\end{equation}
where $\ep$ is a small parameter and $U $ is a smooth function with period $2\pi$ in the variable $t$. In principle $U$ could also depend
on $\ep$ but we have eliminated this dependence for simplicity.
We are interested in the stability properties of $2\pi$-periodic solutions obtained as a continuation from the integrable case
$\ep = 0$. These solutions are understood without collisions and so the restriction to an interval of length $2\pi$, say
$x: [0,2\pi] \to \mathbb{R}^2 \setminus \{0\}$, defines a loop in $\mathbb{R}^2 \setminus \{0\}$. In particular, each $2\pi$-periodic solution has a winding number
$N \in \mathbb{Z}$.
For $\ep = 0$ the system has $2\pi$-periodic solutions with any winding number $N \neq 0$. They are produced by the elliptic (or circular)
orbits with major half-axis $a_N = \vert N \vert^{-2/3}$. This quantity appears as a consequence of Kepler's third law when we look for 
solutions with minimal period $\tfrac{2\pi}{\vert N \vert}$. The sign of $N$ corresponds to the orientation of the orbit. Let 
$\Sigma_N$ denote the set of initial conditions $(x_0,y_0)$ in 
$\left(\mathbb{R}^2 \setminus \{0\}\right) \times \mathbb{R}^2$ producing $2\pi$-periodic solution with winding number $N \neq 0$.
In view of the relationship between the energy and the major axis of a Keplerian ellipse we can describe $\Sigma_N$ by the equations
$$
\frac{1}{2} \vert y_0 \vert^2 - \frac{1}{\vert x_0 \vert} = - \frac{1}{2} N^{2/3}, \qquad N \left(x_{01}y_{02}- x_{02}\,y_{01}\right) > 0.
$$ 
Later we shall see that $\Sigma_N$ is a three dimensional manifold diffeomorphic to $\mathbb{S}^1 \times \mathbb{D}$, where 
$\mathbb{D}$ is the unit open disk.

Let $\phi_t(x_0,y_0)=(x(t;x_0 ,y_0 ),y(t;x_0 ,y_0 ))$ be the flow associated to the the Kepler problem
$$
\dot x = y, \qquad \dot y  = -\frac{x}{\vert x \vert^3}.
$$
Then $\Sigma_N$ is invariant under $\phi_t$ and we can average $U(t,x)$ with respect to the flow over the manifold $\Sigma_N$ to obtain the function
$$
\Gamma_N: \Sigma_N \to \mathbb{R}, \qquad \Gamma_N(x_0,y_0) = \frac{1}{2\pi}\int_0^{2\pi} U\left( t,x(t;x_0 ,y_0 )\right)\,dt.
$$
Any continuation from $\ep = 0$ of $2\pi$-periodic solutions with winding number $N$ must emanate from the set of critical points of $\Gamma_N$.
Conversely, if the critical point satisfies some non-degeneracy condition, such a continuation always exists. This type of result can be obtained using an abstract
variational framework as in \cite{AmbCotEke87} or by a more traditional averaging method. We will follow the second alternative and then it is convenient
to employ a system of coordinates which is natural to the equation for $\ep =0$. Since the Kepler problem is integrable
the action-angle variables seem a natural choice (see \cite{MeyHalOff}). These are the well-known Delaunay variables and they work well for the continuation from positive eccentricity.
 However these coordinates present a blow-up at eccentricity $e = 0$ and they  do not seem suitable to deal with the continuation from circular solutions. Poincar\'{e} proposed
 a variant of the Delaunay variables which solves this difficulty (see \cite{Poi05} and \cite{Fej13}). They are of the form
$(\lambda,\Lambda,\eta ,\xi
)$ with $\lambda \in \mathbb{S}^1$, $\Lambda > 0$ and 
$\eta^2 + \xi^2 < 2\Lambda$.
The sub-manifold $\Sigma_N$ has the simple equation $\Lambda = \vert N \vert^{-2/3}$ and $(\lambda,\eta,\xi)$ are a one-to-one parameterization of
$\Sigma_N$. Poincar\'{e} coordinates are probably the most natural coordinates to deal with the continuation problem \eqref{pk} and they have been already employed in \cite{Yan08}.
In this paper we will show that they are also very useful in the study of the properties of stability. After expressing the function
$\Gamma_N$ in terms of these coordinates, 
$\gamma_N = \gamma_N(\lambda,\eta,\xi)$, we will prove the following 

\begin{theorem}\label{st}
Assume that $\left(\lambda^*,\eta^*,\xi^*\right)$ is a non-degenerate critical point of
$\gamma_N$, i.e., $\nabla \gamma_N = 0$, $\textnormal{det}\left(D^2 \gamma_N \right) \neq 0$
at $\left(\lambda^*,\eta^*,\xi^*\right)$. Then, for any small $\ep > 0$ the bifurcating solution
$x_\ep(t)$ is elliptic if 
$$
\partial^2_{\lambda\lambda} \gamma_N \left(\lambda^*,\eta^*,\xi^*\right) > 0
\quad \mbox{ and } \quad
\textnormal{det}\,D^2 \gamma_N\left(\lambda^*,\eta^*,\xi^*\right) > 0.
$$ 
Conversely, if $\partial^2_{\lambda\lambda} \gamma_N \left(\lambda^*,\eta^*,\xi^*\right)< 0$ or $\textnormal{det}\,D^2 \gamma_N\left(\lambda^*,\eta^*,\xi^*\right) < 0$
then $x_\ep (t)$ is unstable in the Lyapunov sense.
\end{theorem}
We say that a periodic solution $x_{\ep} (t)$ is elliptic if all Floquet multipliers $\mu_i$ satisfy $|\mu_i |=1$, $\mu_i \neq \pm 1$. These multipliers
are the eigenvalues of the monodromy matrix of the linearized periodic system 
\begin{equation}\label{variational}
\ddot{y} +D_x^2 W(t,x_{\ep} (t);\ep)y=0, \qquad W(t,x;\ep)=-\frac{1}{|x|}-\ep U(t,x).
\end{equation}
It must be noticed that the condition of ellipticity is not sufficient to guarantee the stability in the Lyapunov sense of the periodic solution. Our result
can be seen as a first step towards the solution of this delicate stability problem. Typically the known results on linearized stability via the averaging method assume that some of
the multipliers of the unperturbed system are different from $\pm 1$, see for instance \cite{AmbCotEke87} and \cite{Yan08}. However it is well known that the only multiplier of 
the periodic solutions of Kepler problem is $1$.
Indeed the monodromy matrix $P_*$ is given by 
\begin{equation}\label{partialtor}
P_*  =  \left( \begin{array}{cc} I_2 & T \\ 0 & I_2 \end{array}\right), \qquad T =  \left( \begin{array}{cc} \tau & 0 \\ 0 & 0 \end{array}\right), \quad \tau \neq 0.
\end{equation}
To overcome this difficulty with the multipliers we will observe that the matrix $P_*$
enjoys the following symplectic property. Let $\textnormal{Sp}\,(\mathbb{R}^4 )$ be the symplectic group of $4\times 4$ matrices,
then there exists a neighborhood ${\cal U} \subset \textnormal{Sp}\,(\mathbb{R}^4 )$ of $P_*$ such that the spectrum of every matrix $S\in {\cal U}$ satisfies
$$\sigma (S)\subset (\mathbb{R} \setminus \{ 0\} )\cup \mathbb{S}^1 .$$
In other words, the eigenvalues of $S$ can only lie on the real line or on the unit circle. This is an important fact because the monodromy matrix of the
system (\ref{variational}) lies in ${\cal U}$ for small $\ep$ and so the Floquet multipliers can only get out of the unit circle through the real line.
This fact will play a role in the proof of Theorem \ref{st}. Incidentally we note that the previous discussion explains why we work on the plane and not on in 3d.

In practice the conditions of the theorem are not easy to check because the solutions of the Kepler problem
are implicitly defined by the Kepler equation $u-e\sin u =\ell$. In the circular case $u=\ell$ and these solutions are just trigonometric functions and this will allow us
to compute the second derivatives of the Poincar\'e coordinates at $\eta =\xi =0$. In this way we will be led to fully explicit results in the circular case.
For instance,
given $a \in \mathbb{C}$, the forced Kepler problem
$$
\ddot z = - \frac{z}{\vert z \vert^3} - \ep \left( e^{it} + a e^{-it}\right), \qquad z\in \mathbb{C} \setminus \{ 0\},
$$
has two families of $2\pi$-periodic solutions emanating respectively from $e^{it}$ and $-e^{it}$ if  $\vert a \vert \neq 4$. Moreover both families are unstable if $\vert a \vert > 4$ while the first family is elliptic
and the second is unstable if $\vert a \vert < 4$.  
 

The rest of the paper is divided in six sections. Section \ref{s2} deals with certain properties of $4\times 4$ symplectic matrices. The Poincar\'e coordinates are
reviewed in Section \ref{s3}. The next Section \ref{s4} deals with a class of degenerate Hamiltonian systems inspired by the Kepler problem, the typical computations
in the averaging method are presented. The proof of Theorem \ref{st} is presented in Section \ref{s5} and Section \ref{s6} contains the application of the Theorem to the circular case.
Finally we add an appendix with the computations of the derivatives of the Poincar\'e coordinates.

\section{A property of symplectic matrices}\label{s2}

The symplectic group $\textnormal{Sp}\,(\mathbb{R}^{2d})$ is composed by the real matrices $S$ of dimension $(2d) \times (2d)$ satisfying
$$
S^t J S = J,
$$
where $S^t$ is the transpose of $S$ and 
$J = \left( \begin{array}{cc} 0 & I_d \\ - I_d & 0 \end{array}\right)$.

The spectrum of $S$ will be denoted by $\sigma(S)$. It is well known that for symplectic matrices the spectrum is invariant under inversion.
This means that if $\mu$ is an eigenvalue of $S$ then $\mu \neq 0$ and also $\tfrac{1}{\mu}$ is an eigenvalue.

A matrix $S \in \textnormal{Sp}\,(\mathbb{R}^{2d})$ is called \emph{elliptic} if every eigenvalue $\mu \in \sigma(S)$ satisfies 
$\vert \mu \vert = 1$, $\mu \neq \pm 1$. In one degree of freedom ($d=1$) elliptic matrices are easily characterized in terms of the trace. More precisely, 
$S \in \textnormal{Sp}\,(\mathbb{R}^{2})$ is elliptic whenever $\vert \textnormal{tr}\, S \vert < 2$. This property is very useful when dealing with stability problems in the plane. Next we are going to present an analogous criterion in two degrees of freedom. In this case the characterization of elliptic matrices will have a local nature since it will  hold in a neighborhood of a matrix $P$ with the following properties
\begin{equation}\label{parab}
P \in \textnormal{Sp}\,(\mathbb{R}^4), \quad P \neq I, \quad \textnormal{dim}\,\textnormal{ker}\,(P-I) \neq 2, \quad \sigma(P) = \{1\}.
\end{equation}
As an example consider the matrix $P_*$ previously defined by (\ref{partialtor}).

\begin{proposition}\label{discri}
Assume that $P$ is a matrix satisfying \eqref{parab}. Then there exists a neighborhood $\mathcal{U} \subset \textnormal{Sp}\,(\mathbb{R}^4)$
of $P$ such that a matrix $S \in \mathcal{U}$ is elliptic if and only if the following conditions hold
\begin{equation}\label{trace}
\textnormal{det}\,(S-I) > 0 \quad \mbox{ and } \quad \textnormal{tr}\,S < 4.
\end{equation}
\end{proposition}

\begin{remark}
The condition $\textnormal{dim}\,\textnormal{ker}\,(P-I) \neq 2$ is essential. To show this we prove that the above Proposition does not hold if
$$
P =  \left( \begin{array}{cc} \beta^t & 0 \\ 0 & \beta^{-1} \end{array}\right), \qquad \textnormal{ with } \quad \beta =  \left( \begin{array}{cc} 1 & 0 \\ -1 & 1 \end{array}\right).
$$
The matrix $P$ satisfies all the conditions of assumption \eqref{parab} excepting that $\textnormal{dim}\,\textnormal{ker}\,(P-I) = 2$.

For each $\ep \neq 0$ consider the matrix $S_\ep = Q_\ep B_\ep Q_\ep^{-1}$ with
$$
Q_\ep = \left( \begin{array}{cc} A_\ep & M \\ 0 & A_\ep^{-1} \end{array}\right), \qquad A_\ep =  \left( \begin{array}{cc} 1 & 0 \\ 0 & \ep \end{array}\right), \quad M =  \left( \begin{array}{cc} 0 & 0 \\ 0 & 1 \end{array}\right)
$$
and
$$
B_\ep = \left( \begin{array}{cc} C_\ep^t & 0 \\ 0 & C_\ep^{-1} \end{array}\right), \qquad C_\ep =  \left( \begin{array}{cc} 1 & \ep \\ -\ep & 1 \end{array}\right).
$$
The matrices $Q_\ep$ and $B_\ep$ belong to $\textnormal{Sp}\,(\mathbb{R}^4)$ and the inverse of $Q_\ep$ is given by
$Q_\ep^{-1} = \left( \begin{array}{cc} A_\ep^{-1} & -M \\ 0 & A_\ep \end{array}\right)$.
Note that $M$ and $A_\ep$ commute. The spectrum of $S_\ep$ is
$$
\sigma(S_\ep) = \sigma(B_\ep)  = \{ 1 \pm i\ep,\, (1 \pm i\ep)^{-1}\}
$$
and so $S_\ep$ is not elliptic. Moreover 
$\textnormal{det}\,(S_\ep-I) = \textnormal{det}\,(B_\ep-I)  > 0$ and
$\textnormal{tr}\,S_\ep = \textnormal{tr}\,B_\ep < 4$. Therefore the equivalence between ellipticity and \eqref{trace} is broken for $S_\ep$.
Finally we observe that
$$
S_\ep = \left( \begin{array}{cc}\vspace{0.1cm} A_\ep C_\ep^t A_\ep^{-1} & -A_\ep C_\ep^t M + M C_\ep^{-1} A_\ep \\ 0 & A_\ep^{-1} C_\ep^{-1}A_\ep \end{array}\right)
\to P \quad \mbox{ as } \ep \to 0.
$$
\end{remark}

\begin{remark}
A matrix $S \in \textnormal{Sp}\,(\mathbb{R}^{2d})$ is called stable if the sequence $\{ \Vert S^n \Vert\}_{n \in \mathbb{Z}}$ is bounded. Here $\Vert \cdot \Vert$ denotes
 any matrix norm. Most elliptic matrices are stable but there are exceptions if $d \geq 2$. Actually an elliptic matrix is stable if and only if it can be diagonalized. 
 Unstable elliptic matrices can appear in the neighborhood $\mathcal{U}$ given by Proposition 
\ref{discri}. We illustrate this fact with an example concerning the matrix $P_*$ given by \eqref{partialtor}. The notations introduced in the previous remark remain valid. Consider
$$
E_\ep = \left( \begin{array}{cc} R_\ep & \mathcal{T} \\ 0 & R_\ep \end{array}\right)
$$
where
$$
R_\ep =  \left( \begin{array}{cc} \vspace{0.1cm}\cos(\ep^2) & -\sin(\ep^2) \\ \sin(\ep^2) & \cos(\ep^2) \end{array}\right) \quad \mbox{ and }
\quad \mathcal{T} =  \left( \begin{array}{cc} \tau & 0 \\ 0 & -\tau \end{array}\right)
$$ 
with $\ep > 0$ and $\tau \neq 0$. The matrix $E_\ep$ belongs to $\textnormal{Sp}\,(\mathbb{R}^4)$. Moreover it is elliptic if $\ep^2 < \pi$ and it cannot be diagonalized. Hence the same properties hold for $\mathcal{E}_\ep = Q_\ep E_\ep Q_\ep^{-1}$. The matrix $\mathcal{E}_\ep$ converges to $P_*$
as $\ep \to 0$. 
\end{remark}

To prepare the proof of Proposition \ref{discri} we need to introduce some terminology coming from Symplectic Geometry. Given two vectors $\xi,\eta \in \mathbb{R}^4$ we consider the symplectic form acting on them
$$
\omega(\xi,\eta) = \xi^t J \eta.
$$
Sometimes the same form will act on vectors lying in $\mathbb{C}^4$. A Lagrangian plane is a linear manifold $V \subset \mathbb{R}^4$ of dimension two such that the symplectic form vanishes on $V$, that is
$$
\omega(\xi,\eta) = 0 \quad \mbox{ for each } \, \xi,\eta \in V.
$$
The planes $V = \{ \xi \in \mathbb{R}^4 :\; \xi_3 = \xi_4 = 0\} $ and $W = \{ \xi \in \mathbb{R}^4 :\; \xi_1 = \xi_2 = 0\}$ are Lagrangian. The couple $(V,W)$ is an example of a Lagrangian splitting of the space since
$\mathbb{R}^4 = V \oplus W$. Next we present an algebraic result where this notion plays a role.

\begin{lemma}\label{two}
Assume that $1$ is an eigenvalue of the matrix $S \in \textnormal{Sp}\,(\mathbb{R}^4)$ and there is a splitting by invariant Lagrangian planes, that is, there exist two Lagrangian planes $V$ and $W$ satisfying
$$
V \oplus W = \mathbb{R}^4, \quad S(V) = V, \quad S(W) = W.
$$
Then either $S = I$ or $\textnormal{dim}\,\textnormal{ker}\,(S-I) = 2$.
\end{lemma}

\begin{proof}
It follows from Lemma 3.2.4 in the book \cite{MeyHalOff} that $S$ is conjugate in $\textnormal{Sp}\,(\mathbb{R}^4)$ to a matrix of the type
$$
\left( \begin{array}{cc} \beta^t & 0 \\ 0 & \beta^{-1} \end{array}\right)
$$
where $\beta$ in an invertible $2 \times 2$ matrix. From the assumptions on $S$ we deduce that $1$ is also an eigenvalue of $\beta$.
Then either $\beta$ is the identity or $\textnormal{ker}\,(\beta-I)$ has dimension one. In the first case $S=I$ while in the second the kernel of
$S-I$ has dimension two.
\end{proof}

The next result shows that the class of matrices having a splitting by invariant Lagrangian planes is closed in the symplectic group.

\begin{lemma}\label{new}
Let $\{ S_n\}$ be a sequence of matrices in $\textnormal{Sp}\,(\mathbb{R}^4)$ converging to $S$. In addition assume that for each $n \geq 0$ there exists a splitting of $\mathbb{R}^4$ by Lagrangian planes that are invariant under $S_n$. Then there exists another splitting by Lagrangian planes that are invariant under $S$.
\end{lemma}

\begin{proof}
Assume that $(V_n,W_n)$ is a Lagrangian splitting of $\mathbb{R}^4$ with $S_n(V_n) = V_n$, $S_n(W_n) = W_n$. We can apply Lemma 3.2.3. in the book \cite{MeyHalOff}. After selecting an orthonormal basis
of $V_n$, say $\{\alpha_n, \beta_n\}$, we extend it to a symplectic basis of $\mathbb{R}^4$, say $\{\alpha_n,\beta_n,\gamma_n,\delta_n\}$ with
$\gamma_n, \delta_n \in W_n$. In particular
$$
\omega(\alpha_n,\delta_n) = \omega(\beta_n,\gamma_n) = 0, \qquad \omega(\alpha_n,\gamma_n) = \omega(\beta_n,\delta_n)  = 1.
$$
Moreover, from the proof of the Lemma mentioned above,
$$
\omega(\alpha_n,\eta) = \langle \gamma_n,\eta \rangle, \qquad \omega(\beta_n,\eta) = \langle \delta_n, \eta \rangle
$$
for each $\eta \in W_n$. Thus, $\{\gamma_n,\delta_n\}$ is an orthonormal basis of $W_n$. After extracting subsequences we can assume that the vectors $\alpha_n
,\beta_n,\gamma_n,\delta_n$ converge as $n \to +\infty$, say
$$
\alpha_n \to \alpha, \quad \beta_n \to \beta, \quad \gamma_n \to \gamma, \quad \delta_n \to \delta.
$$
The plane $V$ spanned by $\alpha$ and $\beta$ is Lagrangian and the same can be said of the plane $W$ spanned by
$\gamma$ and $\delta$. Note that $\omega(\alpha_n,\beta_n) \to \omega(\alpha,\beta)$.
From $S_n \alpha_n \to S \alpha$, $S_n \beta_n \to S \beta$ we deduce that $S(V) \subset V$.
Since $S$ is an isomorphism, $S(V) = V$. Similarly, $S(W) = W$.
In this way we have found that $S$ admits a splitting by invariant Lagrangian planes. 
\end{proof}

For matrices in $\textnormal{Sp}\,(\mathbb{R}^2)$ the spectrum is contained in $\mathbb{S}^1 \cup (\mathbb{R} \setminus \{0\})$, where
$\mathbb{S}^1 = \{ \mu \in \mathbb{C} \, : \, \vert \mu \vert = 1\}$.
This is not always the case in $\textnormal{Sp}\,(\mathbb{R}^4)$ but we will show that it still holds in the neighborhood of certain matrices.

\begin{lemma}\label{key}
Assume that $P$ satisfies \eqref{parab}. Then there exists a neighborhood $\mathcal{U}_1 \subset \textnormal{Sp}\,(\mathbb{R}^4)$
such that
$$
\sigma(S) \subset \mathbb{S}^1 \cup (\mathbb{R} \setminus \{0\})
$$
for each $S \in \mathcal{U}_1$.
\end{lemma} 

\begin{proof}
By a contradiction argument assume the existence of a sequence of matrices $S_n$ in 
$\textnormal{Sp}\,(\mathbb{R}^4)$ converging to $P$ and such that $\sigma(S_n)$ is not contained in
$\mathbb{S}^1 \cup (\mathbb{R} \setminus \{0\})$. Hence $S_n$ has four different eigenvalues denoted by
$\mu_n, \bar{\mu}_n, \mu_n^{-1}, \bar{\mu}_n^{-1}$ with $\vert \mu_n \vert > 1$ and $\Im \mu_n \neq 0$.
In addition we know that $\mu_n \to 1$ as $n \to +\infty$.

Let $v_n$ and $w_n$ be vectors in $\mathbb{C}^4 \setminus \{0\}$ satisfying 
$$
S_n v_n = \mu_n v_n, \qquad S_n w_n = \mu_n^{-1} w_n.
$$
The plane $V_n \subset \mathbb{R}^4$ spanned by $\Re v_n$ and $\Im v_n$ is invariant under $S_n$.
Similarly the plane $W_n \subset \mathbb{R}^4$ spanned by $\Re w_n$ and $\Im w_n$ satisfies $S_n(W_n) = W_n$.
Since $V_n$ and $W_n$ correspond to different eigenvalues we deduce that they are a splitting of $\mathbb{R}^4$. We claim that $V_n$ and
$W_n$ are Lagrangian. To prove it we observe that the form $\omega$ is preserved by symplectic matrices and so
$$
\omega(v_n, \bar{v}_n) = \omega(S_n v_n, S_n \bar{v}_n) = \mu_n \bar{\mu}_n \omega(v_n, \bar{v}_n).
$$
From $\vert \mu_n \vert \neq 1$ we deduce that $\omega(v_n, \bar{v}_n) = 0$.
This implies that $\omega$ vanishes on $V_n \times V_n$ and so $V_n$ is Lagrangian. The same can be said about $W_n$.

Once we have a splitting by Lagrangian planes
that are invariant under $S_n$ we can apply Lemma \ref{new} to find another Lagrangian splitting invariant under $P$.
Then we can apply Lemma \ref{two} to find a contraction with \eqref{parab}.
\end{proof}

\begin{proof}[Proof of Proposition \ref{discri}]
Given $S \in \textnormal{Sp}\,(\mathbb{R}^4)$ we label its eigenvalues by $\mu_1,\mu_2,\mu_3,\mu_4$ with
$\mu_1\mu_3 = 1$ and $\mu_2 \mu_4 = 1$. Note that eigenvalues are counted according to their multiplicity. 
We define the numbers 
$$
\Delta_1 = \mu_1 + \mu_3, \qquad \Delta_2 = \mu_2 + \mu_4
$$
and observe that $\Delta_1 = \Delta_2 = 2$ when $S = P$. Here we are using the assumption $\sigma(P) = \{1\}$. In consequence 
we can find a neighborhood $\mathcal{U}_2 \subset \textnormal{Sp}\,(\mathbb{R}^4)$ of the matrix $P$ such that
$\Delta_1 > 0$ and $\Delta_2 > 0$ if $S \in \mathcal{U}_2$. Define
$\mathcal{U} = \mathcal{U}_1 \cap \mathcal{U}_2$, where $\mathcal{U}_1$ is given by Lemma \ref{key}. From the definition of $\Delta_1$
and $\Delta_2$, 
$$
\textnormal{tr}\,S = \Delta_1 + \Delta_2.
$$ 
Also
$$
\textnormal{det}\,(S-I) = (2-\Delta_1)(2-\Delta_2).
$$
To prove this identity we observe that the characteristic polynomial of $S$ can be factorized as
$$
\textnormal{det}\,(S-\lambda I ) = (\lambda^2-\Delta_1 \lambda + 1)(\lambda^2 - \Delta_2 \lambda + 1).
$$
Then it is sufficient to let $\lambda = 1$.

We are ready to prove that if $S$ is elliptic then \eqref{trace} holds. The ellipticity of $S$ implies that 
$\mu_3 = \bar{\mu}_1$, $\mu_4 = \bar{\mu}_2$ with $\mu_i \neq 1$ and $\vert \mu_i \vert = 1$. Then
$\Delta_i = 2 \Re \mu_i <2$ and the condition \eqref{trace} becomes a consequence of the previous formulas 
for $\textnormal{tr}\,S$ and $\textnormal{det}\,(S-I)$.

Assume now that $S \in \mathcal{U}$ satisfies \eqref{trace}. It follows from
Lemma \ref{key} that either $\mu_i \in \mathbb{R}$
or $\vert \mu_i \vert = 1$ with $\bar{\mu}_{i+2} = \mu_i$ In any of the two cases the numbers $\Delta_i$ are real and the conditions
in \eqref{trace} are equivalent to
$$
(2-\Delta_1)(2-\Delta_2)>0 \quad \mbox{ and } \quad \Delta_1 + \Delta_2 < 4.
$$
Once we known that $\Delta_i > 0$, $i=1,2$, they become equivalent to $\vert \Delta_i \vert < 2$ and so $\mu_i \in \mathbb{S}^1$, $\mu_i \neq \pm 1$.
\end{proof}

\begin{remark}
It is not hard to check that in $\mathcal{U}$ the conditions 
$$
\textnormal{det}\,(S-I) > 0, \qquad \textnormal{tr}\,S > 4
$$
correspond to the hyperbolic case $\mu_i \in \mathbb{R}$, $\mu_i \neq \pm 1$, $i=1,2$.
Finally,
$$
\textnormal{det}\,(S-I) < 0
$$ 
appears in the mixed elliptic-hyperbolic case, say $\mu_1 \in \mathbb{R}$, $\vert \mu_2 \vert = 1$, $\mu_i \neq \pm 1$, $i=1,2$.
\end{remark}

\section{Poincar\'{e} coordinates}\label{s3}

In connection with Kepler problem
\begin{equation}\label{k}
\dot x = y, \qquad \dot y = -\frac{x}{\vert x \vert^3}, \qquad (x,y)  \in \left(\mathbb{R}^2 \setminus \{0\}\right) \times \mathbb{R}^2,
\end{equation}
we consider the functions
$$
E(x,y) = \frac{1}{2}\vert y \vert^2 - \frac{1}{\vert x \vert} \quad \mbox{ and } \quad M(x,y) = x_1 y_2 - x_2 y_1
$$
and the set
$$
\mathcal{E}_+ = \left\{ (x,y) \in \left(\mathbb{R}^2 \setminus \{0\}\right) \times \mathbb{R}^2 \, : \,
E(x,y) < 0,\, M(x,y) > 0
\right\}.
$$
Elliptic and circular motions with positive orientation have states lying in $\mathcal{E}_+$. In particular the circular motions have states lying on the surface
$$
\mathcal{C} =\left\{ (x,y) \in \left(\mathbb{R}^2 \setminus \{0\}\right) \times \mathbb{R}^2 \, : \,
2E(x,y) M(x,y)^2 = -1
\right\}.
$$
Incidentally we observe that the inequality $2EM^2 \geq -1$ holds everywhere.

The astronomical coordinates $(a,e,l,g)$ are related to the Cartesian coordinates by the formulas
$$
x = R[g] \left( \begin{array}{l}
\vspace{0.2cm}
a(\cos u - e) \\
a \sqrt{1-e^2} \sin u 
\end{array}\right),
\qquad 
y = R[g] \left( \begin{array}{l}
\vspace{0.2cm}
- \frac{\sin u}{a^{1/2} (1-e\cos u)} \\
\frac{\sqrt{1-e^2}\cos u}{a^{1/2}(1-e\cos u)}
\end{array}\right),
$$
where
$$
R[g] = \left( \begin{array}{cc}
\cos g & -\sin g \\
\sin g & \cos g 
\end{array}\right)
\quad \mbox{ and } \quad
u - e \sin u = l.
$$
The elliptic and circular orbits of \eqref{k}, $x=x(t)$, $y=y(t)$, are obtained by letting $l = \tfrac{t}{a^{3/2}}$.

The change of coordinates 
$$
\mathcal{A}: (a,e,l,g) \in \,]0,\infty[\,\times \,]0,1[\,\times \mathbb{T}^2 \mapsto (x,y) \in \mathcal{E}_+ \setminus \mathcal{C} 
$$
is a real analytic diffeomorphism. From now on $\mathbb{T} =\mathbb{R} /2\pi \mathbb{Z}$.  Indeed the map $\mathcal{A}$ is also analytic on 
$\left\{ e = 0\right\}$ but the extension of $\mathcal{A}$ to $\left\{ e \geq 0\right\}$
is no longer a diffeomorphism. The map $\mathcal{A}$ collapses the three-dimensional manifold
$]0,\infty[\,\times \{0\}\times \mathbb{T}^2$ into the surface $\mathcal{C}$.
For further reference we write explicitly the formula of $\mathcal{A}$ on 
$\left\{ e = 0\right\}$, 
\begin{equation}\label{e0}
x =  a \left( \begin{array}{l}
\vspace{0.2cm}
\cos\lambda \\
\sin \lambda 
\end{array}\right),
\quad
y =  a^{-1/2} \left( \begin{array}{l}
\vspace{0.2cm}
-\sin\lambda \\
\cos \lambda 
\end{array}\right),
\quad
\lambda = l + g.
\end{equation}

Next we introduce the Delaunay variables $(l,L;g,G)$ by the formulas
$$
L = \sqrt{a}, \qquad G = \sqrt{a(1-e^2)}.
$$
If we define
$$
\Sigma = \left\{ (L,G) \in \mathbb{R}^2 \, : \,0 < G < L \right\}
$$
then the map
$$
\mathcal{D}: (L,G,l,g) \in \Sigma \times \mathbb{T}^2 \mapsto (x,y) \in \mathcal{E}_+ \setminus \mathcal{C}
$$
is also an analytic diffeomorphism. The advantage with respect to the previous coordinates is that $\mathcal{D}$ defines a symplectic transformation; that is,
$$
dl \wedge dL + dg \wedge dG = dx_1 \wedge dy_1 + dx_2 \wedge dy_2.
$$
The problem with the circular motions still remains in these new coordinates, now
$\left\{ e = 0\right\}$ is replaced by $\left\{ L = G \right\}$. To overcome this difficulty, 
Poincar\'{e} coordinates $(\lambda,\Lambda;\eta ,\xi )$ are introduced by the formulas
$$
\lambda = l + g, \qquad \Lambda = L, \qquad \eta = \sqrt{2H} \sin h, \qquad \xi = \sqrt{2H} \cos h,
$$
where $h = -g$ and $H = L - G$.

A straightforward computation shows that
\begin{equation}\label{sym}
d\lambda \wedge d\Lambda + d\eta \wedge d\xi = dl \wedge dL + dg \wedge dG.
\end{equation}
The above discussion follows along the lines of the original exposition by Poincar\'e in Sections 56-57 of \cite{Poi05}. The notation is taken from an unpublished report
by A. Chenciner. In the book by Poincar\'e the coordinates $H$ and $h$ are called $\rho_1$ and $\omega_1$.
In principle the formula (\ref{sym}) is valid on the domain
$$
\Omega^* = \left\{ (\lambda,\Lambda;\eta,\xi) \in \mathbb{T} \times \,]0,\infty[\,\times \mathbb{R}^2 \, : \; 0 < \eta^2 + \xi^2 < 2\Lambda\right\}
$$
and so the map 
$$
\mathcal{P}: (\lambda,\Lambda,\eta,\xi) \in \Omega^* \mapsto (x,y) \in \mathcal{E}_+ \setminus \mathcal{C}
$$
is a symplectic diffeomorphism.

This map has a continuous extension defined on
$$
\Omega = \left\{ (\lambda,\Lambda;\eta,\xi) \in \mathbb{T} \times \,]0,\infty[\,\times \mathbb{R}^2 \, : \, \eta^2 + \xi^2 < 2\Lambda\right\}
$$
and, in contrast to the previous cases, the manifold $\Omega \setminus \Omega^*$ has the same dimension of
$\mathcal{C}$. The map $\mathcal{P}$ on $\Omega \setminus \Omega^*$ is given by the formulas in \eqref{e0}
with $a = \Lambda^2$. It defines a diffeomorphism between $\Omega \setminus \Omega^*$ and $\mathcal{C}$. From the formulas it is not obvious at all that the extended map is analytic in the variables $(\lambda,\Lambda,\eta,\xi)$. 
Poincar\'e proved this in Sections  65-69 of his book \cite{Poi05}. For further reference we state the main conclusion of this Section as a proposition.

\begin{proposition}\label{analytic}
The map $\mathcal{P}: \Omega \to \mathcal{E}_+$ is a real analytic symplectic diffeomorphism.
\end{proposition}

For completeness a proof of this result will be included in the Appendix. Also in the Appendix we will  explain how to compute the successive derivatives
of $\mathcal{P} = (x,y)$ at $\eta = \xi = 0$. To obtain explicit formulas in the circular case we need to know the derivatives of $x$ of first and second order with respect to $\lambda,\eta$ and $\xi$. These derivatives are
$$
\partial_{(\lambda,\eta,\xi)} x(\lambda,\Lambda,0,0) = 
\left(\Lambda^2 i e^{i\lambda},\frac{\Lambda^{3/2}}{2}\left(3i+ie^{2i\lambda}\right),
\frac{\Lambda^{3/2}}{2}\left(-3 + e^{2i\lambda}\right)\right)
$$
and
$$
\partial^2_{(\lambda,\eta,\xi)} x(\lambda,\Lambda,0,0) = 
\left( 
\begin{array}{ccc}\vspace{0.2cm}
-\Lambda^2 e^{i\lambda} & -\Lambda^{3/2} e^{2i\lambda} & \Lambda^{3/2} i e^{2i\lambda}  \\
\vspace{0.2cm}
* & -\Lambda e^{i\lambda} - \tfrac{\Lambda}{4}\left( e^{-i\lambda} + 3e^{3i\lambda}\right) &
\tfrac{\Lambda}{4} \left( -i e^{-i\lambda} + 3 i e^{3i\lambda}\right) \\
* & * &  -\Lambda e^{i\lambda} + \tfrac{\Lambda}{4}\left( e^{-i\lambda} + 3 e^{3i\lambda}\right)
\end{array}
\right). 
$$
To simplify the formulas we use complex notation and identify $x = (x_1,x_2)$ to $x = x_1 + i x_2$.

\section{Local continuation of periodic solutions}\label{s4}

The system \eqref{pk} has a Hamiltonian structure with
$$
\dot x = K_y, \qquad \dot y = -K_x,
$$
and
$$
K(t, x, y; \ep) = \frac{1}{2}\vert y \vert^2 - \frac{1}{\vert x \vert} - \ep \,U(t,x).
$$
From now on we assume that $U: \mathbb{T} \times \mathbb{R}^2 \to \mathbb{R}$ is a continuous function having partial derivatives with respect to $x$ up to the third order. 
Moreover the functions $(t,x) \mapsto \partial_x^\alpha U(t,x)$ with $\vert \alpha \vert \leq 3$ are continuous. These assumptions will be summarized by $U \in C^{0,3}(\mathbb{T} \times \mathbb{R}^2)$.

In the elliptic region $\mathcal{E}_+$ this system can be transformed via the Poincar\'e coordinates into a
new system with Hamiltonian function
$$
\mathcal{K}_\ep (t,\lambda,\Lambda,\eta,\xi) = -\frac{1}{2\Lambda^2} - \ep\,\mathcal{U}(t,\lambda,\Lambda,\eta,\xi)
$$
where $(\lambda,\Lambda,\eta,\xi)$ lies in $\Omega$ and $\mathcal{U}(t,\lambda,\Lambda,\eta,\xi)
= U(t,x(\lambda,\Lambda,\eta,\xi))$. Note that $x = x(\lambda,\Lambda,\eta,\xi)$ is defined by 
$\mathcal{P}$. In agreement with (\ref{sym}) the system becomes 
\begin{equation}\label{hp}
\dot \lambda = \frac{1}{\Lambda^3} - \ep\, \partial_{\Lambda}\,\mathcal{U},
\qquad \dot \Lambda = \ep \partial_{\lambda}\, \mathcal{U},
\qquad \dot \eta = -\ep\, \partial_{\xi}\, \mathcal{U}
\qquad \dot \xi =\ep\, \partial_{\eta}\, \mathcal{U},
\end{equation}
and we observe that for $\ep = 0$ and $N \geq 1$ this system has a continuum of $2\pi$-periodic
solutions defined by
$$
\lambda(t) = Nt + \lambda_0, \quad \Lambda(t) = \frac{1}{N^{1/3}},
\quad \eta(t) = \eta_0, \quad \xi(t) = \xi_0.
$$
Note that $\lambda$ is an angular variable and the minimal period of these solutions is 
$\tfrac{2\pi}{N}$.

The perturbed system \eqref{hp} is in the framework of averaging theory for periodic
systems. We consider
the function
$$
\gamma_N(\lambda_0,\eta_0,\xi_0) = \frac{1}{2\pi} \int_0^{2\pi}
\mathcal{U} \left(t,Nt + \lambda_0,\frac{1}{N^{1/3}},\eta_0,\xi_0\right)\,dt
$$
defined on the solid torus $\lambda_0 \in \mathbb{T}$, $\eta_0^2 + \xi_0^2<\tfrac{2}{N^{1/3}}$.
The next result establishes a link between this function and the continuation of
$2\pi$-periodic solutions.

\begin{proposition}\label{exis}
Assume that $(\lambda_0,\eta_0,\xi_0)$ is a non-degenerate critical point of
$\gamma_N$. Then there exists a $\mathcal{C}^1$ family of $2\pi$-periodic solutions of \eqref{hp},
$\lambda = \lambda(t,\ep)$, $\Lambda = \Lambda(t,\ep)$, 
$\eta = \eta(t,\ep)$, $\xi = \xi(t,\ep)$,
defined for $\ep$ small and satisfying
$$
\lambda(t,0) = Nt + \lambda_0, \quad \Lambda(t,0) = \frac{1}{N^{1/3}},
\quad \eta(t,0) = \eta_0, \quad \xi(t,0) = \xi_0.
$$
\end{proposition}

\begin{remark}
1. The point $(\lambda_0,\eta_0,\xi_0)$ satisfies $\nabla \gamma_N (\lambda_0,\eta_0,\xi_0) = 0$ and also
$\textnormal{det}\, \left( D^2 \gamma_N (\lambda_0,\eta_0,\xi_0)\right) \neq 0$.
If we pass to Cartesian coordinates the corresponding periodic solution
$x(t,\ep)$ has winding number $N$. This is clear for $\ep = 0$ and then we observe that the winding number is independent
of $\ep$.
\smallbreak
\noindent
2. To connect $\gamma_N$ with the function $\Gamma_N$ defined in the introduction we observe that the map
$\mathcal{P}$ defines a diffeomorphism between $\Omega \cap \left\{ \Lambda = \tfrac{1}{N^{1/3}}\right\}$
(the domain of $\gamma_N$) and $\Sigma_N$ as defined in the introduction. Then $\Gamma_N \circ \mathcal{P} = \gamma_N$
and so the critical points of both functions are in a one-to-one correspondence. In addition the Morse index coincides.
\smallbreak
\noindent
3. The result is also valid for negative $N$. Indeed it is sufficient to change the direction of time in \eqref{hp}.
\end{remark}

The above result is a consequence of the general theory of averaging but we will present a direct proof. This will allow us to perform some computations that will be useful later.

The Kepler problem can be seen as a model for a large class of degenerate Hamiltonian systems and we find convenient to work in a general setting. From now on we consider a system of symplectic coordinates $(\theta,q;r,p)$ associated to the $2$-form
$$
d\theta \wedge dr + \sum_{i=1}^d dq_i \wedge dp_i.
$$
Here $\theta$ is an angular variable, the points $q = (q_1,\ldots,q_d)$,
$p = (p_1,\ldots,p_d)$ lie in $\mathbb{R}^d$ and the phase space is of the type
$$
\Omega = \left\{ z = (\bar\theta,q,r,p) \, : \, \bar\theta \in \mathbb{T}, r \in I, (q,p) \in D_r \right\}
$$
where $I$ is an open interval in $\mathbb{R}$ and $D_r$ is an open subset of $\mathbb{R}^d \times \mathbb{R}^d$ which may change 
continuously with $r$.

Consider also the Hamiltonian function
$$
\mathcal{K}_\ep(t,\theta,q,r,p) = h(r) - \ep \mathcal{U}(t,\theta,q,r,p)
$$
with $h \in C^3(I)$ and $\mathcal{U} \in C^{0,3}(\mathbb{T} \times \Omega)$.

The associated system is
$$
\dot z = J \nabla \mathcal{K}_\ep (t,z)
$$
or, in coordinates, 
\begin{equation}\label{gkp}
\dot\theta = h'(r) - \ep \partial_r \,\mathcal{U}, \qquad
\dot q = - \ep \partial_p \,\mathcal{U}, \qquad
\dot r = \ep \partial_\theta \,\mathcal{U}, \qquad
\dot p = \ep \partial_q \,\mathcal{U}.
\end{equation}
The perturbed Kepler problem \eqref{hp} is recovered with $\theta = \lambda$, $r = \Lambda$, $q=\eta$, $p = \xi$ and
$h(r) = -\tfrac{1}{2}r^{-2}$.

Given an integer $N$ we select $r_N \in I$ such that
\begin{equation}\label{H1}
h'(r_N) = N
\end{equation}
The number $r_N$ (when it exists) produces a family of $2\pi$-periodic solutions of the system \eqref{gkp} for $\ep = 0$, namely
\begin{equation}\label{fam}
\theta(t) = \theta_0 + Nt, \quad q(t) = q_0, \quad r(t) = r_N, \quad p(t) = p_0
\end{equation}
with parameters $\bar\theta_0 \in \mathbb{T}$ and $(q_0,p_0) \in D_{r_N}$.

From now on the initial conditions will be simply denoted by $(\theta,q,r,p)$. The Poincar\'e map of \eqref{gkp} is denoted by
$$
\Pi_\ep: (\theta,q,r,p) \mapsto (\theta',q',r',p')
$$
and sends an initial condition at time $t = 0$ to the value of the corresponding solution at time
$t = 2\pi$. For $\ep = 0$ this map is well defined on the whole domain $\Omega$ and can be computed explicitly,
\begin{equation}\label{p0}
\Pi_0(\theta,q,r,p)  = \left(\theta + 2\pi h'(r),q,r,p\right).
\end{equation}

Once a point $(\theta_*,q_*,r_*,p_*)$ lying in $\Omega$ has been selected, the map $\Pi$ is well defined and of class $C^2$ in a neighborhood of
$(\theta_*,q_*,r_*,p_*,\ep = 0)$.

We want to find periodic solutions emanating from the family \eqref{fam} at a point satisfying $r_* = r_N$. From now on we work on a small neighborhood of this point and try to solve the system
\begin{equation}\label{fsys}
\theta' = \theta + 2\pi N, \qquad q' = q, \qquad r' = r, \qquad p' = p.
\end{equation}

Let us assume that
\begin{equation}\label{H2}
h''(r_N) \neq 0.
\end{equation}
Then $\tfrac{\partial \theta'}{\partial r}$ does not vanish at $\epsilon =0$. The implicit function theorem can be applied to the equation
$\theta' = \theta + 2\pi N$ in order to express the unknown $r$ as a $C^2$ function of the remaining variables, 
$$
r = \varphi(\theta,q,p;\ep) \qquad \mbox{ with } \; \varphi(\theta,q,p;0) = r_N.
$$
Then the system \eqref{fsys} is reduced to
$$
q' = q, \qquad r' = r, \qquad p' = p
$$
where $r$ is now a dependent variable. In other words, we are looking for zeros of the map
$\sigma = (\sigma_1,\sigma_2,\sigma_3)$, 
$$
\sigma(\theta,q,p;\ep) = (q'-q,r'-r,p'-p).
$$
From the explicit formula for $\Pi_0$ given by \eqref{p0} and the identity $r = r_N$ for $\ep = 0$ we deduce that the restriction
of $\sigma$ to $\ep = 0$ vanishes everywhere. To obtain a bifurcation for $\ep \neq 0$ we must look for zeros of the map
$$
\Phi(\theta,q,p) = \partial_\ep \sigma(\theta,q,p;0).
$$
By chain rule,
$$
\begin{array}{l}
\vspace{0.3cm}
\displaystyle{\partial_\ep \sigma_1 = \frac{\partial q'}{\partial \ep} + \frac{\partial q'}{\partial r}\frac{\partial r}{\partial \ep}}  \\
\vspace{0.3cm}
\displaystyle{\partial_\ep \sigma_2 = \frac{\partial r'}{\partial \ep} + \frac{\partial r'}{\partial r}\frac{\partial r}{\partial \ep} - \frac{\partial r}{\partial \ep}} \\
\displaystyle{\partial_\ep \sigma_3 = \frac{\partial p'}{\partial \ep} + \frac{\partial p'}{\partial r}\frac{\partial r}{\partial \ep}}. 
\end{array}
$$
Again from \eqref{p0} it follows that $\tfrac{\partial r'}{\partial r} = 1$ and
$\tfrac{\partial q'}{\partial r} = \tfrac{\partial p'}{\partial r} = 0$ at $\ep = 0$.
Hence,
$$
\Phi = \left(\frac{\partial q'}{\partial \ep},\frac{\partial r'}{\partial \ep},\frac{\partial p'}{\partial \ep} \right)\vert_{\ep = 0} =
\pi \circ \partial_\ep \Pi_\ep \vert_{\ep = 0},
$$
where $\pi$ is the projection defined by
$$
\pi(\theta,q,r,p) = (q,r,p).
$$
To compute $\partial_\ep \Pi_\ep \vert_{\ep = 0}$ we first consider the variational equation
(with respect to $\ep$) for the system \eqref{gkp} around a solution
of the family \eqref{fam} with $\ep = 0$,
$$
\dot \delta_1 = h''(r_N) \delta_3 - \partial_r \, \mathcal{U}, \qquad
\dot \delta_2 = -\partial_{p}\, \mathcal{U},  \qquad
\dot \delta_3 = \partial_{\theta}\, \mathcal{U}, \qquad
\dot \delta_4 = \partial_{q} \,\mathcal{U}.
$$
The last three coordinates are easily computed by direct integration and,
after imposing the initial condition $\delta_i(0) = 0$, $1\leq i \leq 4$, we are led to the identity
\begin{equation}\label{sch}
\Phi = 2\pi \mathcal{R} \,\nabla \mathcal{U}_{\#}  
\end{equation}
where
$$
\mathcal{R} = \left(
\begin{array}{ccc}
0 & 0 & -I_d \\
1 & 0 & 0 \\
0 & I_d & 0
\end{array}
\right)
\quad \mbox{ and }
\quad
\mathcal{U}_{\#}(\theta,q,p) = \frac{1}{2\pi}\int_0^{2\pi} \mathcal{U}(t,\theta+Nt,q,r_N,p)\,dt.
$$
A sufficient condition for the bifurcation of $\sigma = 0$ is $\Phi = 0$ and $\textnormal{det}\,\Phi' \neq 0$.
This is equivalent to finding non-degenerate critical points of $\mathcal{U}_{\#}$. In many aspects the previous analysis is reminiscent of
Section 46 in \cite{Poi92}. 

The conditions \eqref{H1} and \eqref{H2} are satisfied by the Kepler problem with $r_N = \tfrac{1}{N^{1/3}}$. 
Since $\gamma_N = \mathcal{U}_{\#}$ in this case, the proof of Proposition \ref{exis} is complete.

It is possible to obtain other bifurcation results with more relaxed conditions on the critical point. For instance, a bifurcation exists when the function
$\gamma_N$ has an isolated critical point and $\nabla \gamma_N$ has non-zero Brouwer index at this point. This bifurcation does not necessarily leads to a smooth family of periodic solutions.

\section{Periodic solutions of elliptic type}\label{s5}

In this Section we prove the main result of the paper, Theorem \ref{st}. First we work with the general class of systems introduced in Section \ref{s4}.
This will allow us to obtain an instability result valid in arbitrary dimension.

Let $z(t,\ep) = (\theta(t,\ep),q(t,\ep),r(t,\ep),p(t,\ep))$ be a $C^1$ family of $2\pi$-periodic solutions of the system \eqref{gkp}.
We assume that the conditions \eqref{H1} and \eqref{H2} hold and
$$
z(t,0) = (\theta_* + Nt, q_*, r_N, p_*).
$$

\begin{proposition}\label{ins}
In the previous setting assume that one of the conditions below hold,
\begin{equation}\label{H3}
h''(r_N) \partial_{\theta\theta}^2 \,\mathcal{U}_{\#}(\theta_*,q_*,p_*) > 0,
\end{equation}
\begin{equation}\label{H4}
h''(r_N) \,\textnormal{det}\,\left[D^2 \mathcal{U}_{\#}(\theta_*,q_*,p_*)\right] > 0,
\end{equation}
then, for small $\ep > 0$, the solution $z(t,\ep)$ is unstable in the Lyapunov sense.
\end{proposition}

The proof of this result will provide some additional information on the variational equation
\begin{equation}\label{lvs}
\dot \delta = J \nabla \mathcal{K}_\ep (t,z(t,\ep))\delta.
\end{equation}
 We will prove that some Floquet multiplier must lie outside the unit circle. This implies that $z(t,\ep)$ is Lyapunov unstable and also that it is not elliptic
\begin{proof}
The Poincar\'e map associated to \eqref{gkp} is denoted by $\Pi_\ep(z) = \Pi(z,\ep)$. The monodromy matrix associated to \eqref{lvs}
is given by
$$
S(\ep) = \partial_z \Pi(z(0,\ep),\ep).
$$
This matrix defines a path in the symplectic group $\textnormal{Sp}\,(\mathbb{R}^{2(d+1)})$. We will compute the Taylor expansion of $S(\ep)$ 
at $\ep = 0$ but first we go back to the formula \eqref{p0} in order to compute the successive derivatives of $\Pi(z,0)$.
The Jacobian matrix is
\begin{equation}\label{expd}
\partial_z \Pi(z,0) =  \left(
\begin{array}{cccc}
1 & 0 & \tau & 0 \\ 
0 & I_d & 0 & 0 \\
0 & 0 & 1 & 0 \\
0 & 0 & 0 & I_d
\end{array}
\right) \qquad \mbox{ with } \; \tau = 2\pi h''(r),
\end{equation}
and all second derivatives vanish excepting $\tfrac{\partial^2 \theta'}{\partial r^2} = 2\pi h'''(r)$. 

Differentiating $S(\ep)$ with respect to $\ep$ and evaluating at $\ep = 0$,
$$
S'(0) = H = \partial^2_{\ep z} \Pi(z_*,0) +\mathcal{N}
$$
where $z_* = (\theta_*,q_*,r_N,p_*)$ and $\mathcal{N} = (n_{ij})$ satisfies $n_{ij} = 0$ if $(i,j) \neq (1,d+2)$ and
$$
n_{1, d+2} =\frac{\partial^2 \theta'}{\partial r^2}(z_*,0) \frac{\partial r}{\partial \ep}(0,0).
$$
Note that this coefficient occupies the same position at $\mathcal{N}$ as the coefficient $\tau$ in the above matrix.

Thus 
\begin{equation}\label{Texp}
S(\ep) = P_* + \ep H + o(\ep), \quad \mbox{ as } \, \ep \to 0,
\end{equation}
with $P_* = \partial_z \Pi(z_*,0)$. From the previous expansion it is not difficult to see that
$$
\textnormal{det}\,(S(\ep)-I) = (-1)^{d+1} \tau_N \ep^{2d+1} \textnormal{det}\, M + o\left(\ep^{2d+1}\right)
$$
where $\tau_N = 2\pi h''(r_N)$ and $M$ is the sub-matrix of $\partial^2_{\ep z} \Pi(z_*,0)$ obtained after eliminating the first row and the column $d+2$.

To determine $M$ we first define the map
$$
\tilde \Pi(\theta,q,p,\ep) = (q',r',p')
$$
where $(\theta',q',r',p') = \Pi_\ep(\theta,q,r_N,p)$. Some computations of the previous Section leading to
\eqref{sch} show that
$$
\partial_\ep \tilde\Pi(\theta,q,p,0) = 2\pi\mathcal{R} \nabla \mathcal{U}_{\#} (\theta,q,p).
$$
Differentiating this identity with respect to $\theta$, $q$ and $p$ and evaluating at 
$(\theta_*,q_*,p_*)$ we obtain
$$
M = 2\pi \mathcal{R} D^2 \mathcal{U}_{\#}(\theta_*,q_*,p_*).
$$
Since $\textnormal{det}\,\mathcal{R} = (-1)^d$ we deduce that
\begin{equation}\label{F1}
\textnormal{det}\,(S(\ep)-I) = -\tau_N (2\pi)^{2d+1} \ep^{2d+1} \textnormal{det}\,\left[ D^2 \mathcal{U}_{\#}(\theta_*,q_*,p_*)\right] + o\left(\ep^{2d+1}\right).
\end{equation}
If we assume that \eqref{H4} holds then the determinant of $S(\ep) - I$ is negative if $\ep > 0$ is small enough.
Let $\mu_1(\ep),\ldots,\mu_{2d+2}(\ep)$ be the eigenvalues of $S(\ep)$. They satisfy $\mu_k(\ep) \to 1$ and $\mu_k(\ep) \neq 1$ if
$\ep \neq 0$ is small. We prove by a contradiction argument that some $\mu_k(\ep)$ must lie outside $\mathbb{S}^1$, for otherwise
the eigenvalues could be labeled in such a way that $\overline{\mu_i(\ep)} = \mu_{i+d+1}(\ep)$ and
$$
\textnormal{det}\,(S(\ep)-I) = \prod_{i=1}^{d+1} \vert 1 - \mu_i(\ep) \vert^2 > 0.
$$
This proves the Proposition when \eqref{H4} holds. To deal with the other assumption \eqref{H3} we go back to the expansion \eqref{Texp} and compute the trace
$$
\textnormal{tr}\,S(\ep) = 2d+2 + \ep \textnormal{tr}\,H  + o(\ep) = 2d+2 + \ep \textnormal{tr}\, \partial^2_{\ep z} \Pi(z_*,0) + o(\ep).
$$
The Poincar\'e map is canonical and so
$$
d\theta' \wedge dr' + \sum_{i=1}^d dq_i' \wedge dp_i' = d\theta \wedge dr + \sum_{i=1}^d dq_i \wedge dp_i,
$$
implying
$$
\frac{\partial \theta'}{\partial \theta} \frac{\partial r'}{\partial r} -
\frac{\partial \theta'}{\partial r} \frac{\partial r}{\partial \theta} +
\sum_{j=1}^d\left( \frac{\partial q_j'}{\partial \theta} \frac{\partial p_j'}{\partial r} -
\frac{\partial q_j'}{\partial r} \frac{\partial p_j'}{\partial \theta} \right)
 = 1 
$$
and
$$
\frac{\partial \theta'}{\partial q_i} \frac{\partial r'}{\partial p_i} -
\frac{\partial \theta'}{\partial p_i} \frac{\partial r}{\partial q_i} +
\sum_{j=1}^d\left( \frac{\partial q_j'}{\partial q_i} \frac{\partial p_j'}{\partial p_i} -
\frac{\partial q_j'}{\partial p_i} \frac{\partial p_j'}{\partial q_i} \right)
 = 1. 
$$
Differentiating these identities with respect to $\ep$ and evaluating at $\ep = 0$,
$$
\partial^2_{\ep \theta} \theta' + \partial^2 _{\ep r} r' - \frac{\partial \theta'}{\partial r}\partial^2_{\ep\theta}r' = 0
$$
and
$$
\partial^2_{\ep q_i} q_i' + \partial^2 _{\ep p_i} p_i' = 0.
$$
Here we are using the formula \eqref{expd}. In particular,
$$
\frac{\partial \theta'}{\partial r}|_{\ep = 0} = 2\pi h''(r).
$$
In consequence, if $\ep = 0$
$$
\partial^2_{\ep \theta} \theta' + \partial^2 _{\ep r} r' + \sum_{i=1}^d \left( \partial^2_{\ep q_i} q_i' + \partial^2 _{\ep p_i} p_i' \right)
= 2\pi h''(r) \partial^2_{\ep \theta} r'.
$$
We know from the previous computations on the matrix $M$ that 
$$
\partial^2_{\ep \theta} r' = 2\pi \partial^2_{\theta\theta}\, \mathcal{U}_{\#} \quad \mbox{ if } \, \ep = 0, \, r = r_N.
$$ 
From the previous expansion for the trace we deduce now that
\begin{equation}\label{F2}
\textnormal{tr}\,S(\ep) = 2d+2 + \ep 4\pi^2 h''(r_N) \partial^2_{\theta\theta}\, \mathcal{U}_{\#}(\theta_*,q_*,p_*) + o(\ep).
\end{equation}
Therefore $\textnormal{tr}\,S(\ep) > 2d+2$ if \eqref{H2} holds and $\ep > 0$ is small enough. We claim that some eigenvalues must lie outside 
$\mathbb{S}^1$. Otherwise we could label the eigenvalues as before and deduce that
$$
\textnormal{tr}\,S(\ep) = 2 \sum_{i=1}^d \Re (\mu_i(\ep)) \leq 2(d+1).
$$
Note that now $\mu_i(\ep) = 1$ is admissible.
\end{proof}

The previous Proposition can be applied to the Kepler problem. It leads to the conclusions of the main Theorem \ref{st}
on instability. To obtain the conclusion on ellipticity the dimension will play a role. Since the monodromy matrix $S(\ep)$ now belongs to $\textnormal{Sp}\,(\mathbb{R}^4)$, the results of Section
\ref{s2} can be employed. From the formula \eqref{expd} we deduce that $S(0) = P_*$ where $P_*$ is given by
\eqref{partialtor} with $\tau = \tau_N = -6\pi N^{4/3}$. Then Proposition \ref{discri} implies that, 
for $\ep > 0$ small, the matrix $S(\ep)$ will be elliptic if
$$
\textnormal{det}\,(S(\ep)-I) > 0 \quad \mbox{ and } \quad \textnormal{tr}\,S(\ep) < 4.
$$
The expansions \eqref{F1} and \eqref{F2} imply that these inequalities are equivalent to
$$
\textnormal{det}\,[D^2 \mathcal{U}_{\#}] > 0 \quad \mbox{ and } \quad \partial^2_{\theta\theta}\, \mathcal{U}_{\#} > 0
$$
with $\mathcal{U}_{\#} = \gamma_N$.

\section{Explicit computations for the circular case}\label{s6}

In this section we will discuss the continuation from circular orbits for the forced Kepler problem
\begin{equation}\label{2}
\ddot x = -\frac{x}{\vert x \vert^3} - \ep\, p(t)
\end{equation}
where $p: \mathbb{R} \to \mathbb{R}^2$, $p = (p_1,p_2)^t$, is a continuous and $2\pi$-periodic function.

The associated potential is 
$$
U(t,x) = -\langle p(t), x \rangle
$$
and, after fixing an integer $N > 0$, the function $\gamma_N$ introduced in Section \ref{s4} can be expressed as
$$
\gamma_N(\lambda,\eta,\xi) = -\frac{1}{2\pi} \int_0^{2\pi}
\left\langle p(t), x(\lambda + Nt, \Lambda_N,  \eta ,\xi ) \right\rangle \,dt
$$
where $\Lambda_N = \tfrac{1}{N^{1/3}}$ and $x = x(\lambda,\Lambda,\eta,\xi)$ is defined by the Poincar\'e coordinates.
The domain of $\gamma_N$ is the solid torus $\mathbb{T} \times \mathbb{B}_N$ with
$$
\mathbb{B}_N = \{ (\eta,\xi) \in \mathbb{R}^2 : \, \eta^2 + \xi^2 < 2 \Lambda_N \}.
$$ 
In order to apply the results of Sections \ref{s4} and \ref{s5} at the circular orbits we must look for critical points of $\gamma_N$ lying on the equator $\mathbb{T} \times \{0\}$. 
It will be convenient to employ the complex notation that is typical in Fourier analysis. To this end we identify $p = p_1 + i p_2$, $x = x_1 + i x_2$ and rewrite $\gamma_N$ by the expression
$$
\gamma_N(\lambda,\eta,\xi) = -\frac{1}{2\pi}
\int_0^{2\pi} \Re\left( p(t) \,\overline{x(\lambda+Nt,\Lambda_N,\eta,\xi)}\right) \,dt. 
$$
The Fourier coefficients of $p(t)$ are given by the formulas 
$$
c_n = \frac{1}{2\pi} \int_0^{2\pi} p(t) e^{-int}\,dt, \qquad n \in \mathbb{Z}.
$$
Using the identities presented at the end of Section \ref{s3} we can express the gradient of $\gamma_N$ along $\mathbb{T} \times \{0\}$ in terms of $c_0, c_N$ and $c_{2N}$. In particular 
$$
\nabla \gamma_N(\lambda,0,0) = 0
$$
is equivalent to
\begin{equation}\label{pregrad}
\Im (e^{-i\lambda} c_N )=0,\qquad e^{-2i\lambda} c_{2N} =3\overline{c_0} .
\end{equation}
This system in $\lambda$ has a solution if and only if
\begin{equation}\label{grad}
3 \vert c_0 \vert = \vert c_{2N} \vert \quad \mbox{ and } \quad
c_0 c_{2N} \overline{c_N}^2 \in [0,\infty[\,.
\end{equation}
When some of these Fourier coefficients do not vanish this condition implies that $\gamma_N$ has exactly two critical points on 
$\mathbb{T} \times \{0\}$ and they are antipodal, that is, 
$(\lambda^*,0,0)$ and $(\lambda^* + \pi,0,0)$ for some $\lambda^*$.
In the degenerate case $c_0 = c_N = c_{2N} = 0$ the function $\gamma_N$ has a continuum of critical points along the equator.

In the real Banach space $C(\mathbb{T},\mathbb{R}^2)$ the subset described by \eqref{grad} contains a manifold of codimension two ($c_0 \neq 0$, $c_N \neq 0$) but it also contains a linear manifold of codimension $4$ described by 
\begin{equation}\label{lm}
c_0 = c_{2N} = 0.
\end{equation}
For simplicity in the computations we will restrict to this linear manifold. From now on we assume that \eqref{lm} holds. We also assume that
$$
c_N \neq 0,
$$
for otherwise critical points on 
$\mathbb{T} \times \{0\}$ would not be isolated.

Let us choose $\lambda^*$ such that
$$
e^{i\lambda^*} = \frac{c_N}{\vert c_N \vert}.
$$
In agreement with \eqref{pregrad} we deduce that $(\lambda^*,0,0)$ and
$(\lambda^*+\pi,0,0)$ are critical points. From the identities for the second derivatives of $x$ (Section \ref{s3}), we deduce that
$$
D^2 \gamma_N(\lambda^*,0,0) = 
\left( \begin{array}{c|c}
\Lambda_N^2 \vert c_N \vert & 0 \qquad 0  \\
\hline
\begin{array}{l} 0 \\ 0 \end{array} & \Lambda_N \mathcal{M}(p)   
\end{array}\right)
$$
with
$$
\mathcal{M}(p) = \left( \begin{array}{cc}
\vspace{0.2cm}
\displaystyle{\vert c_N \vert + \frac{1}{4} \Re\left( \frac{c_N c_{-N}}{\vert c_N \vert} + 3 \frac{\overline{c_N}^3}{\vert c_N \vert^3}c_{3N} \right)} & 
\displaystyle{\frac{1}{4} \Im\left(\frac{c_N c_{-N}}{\vert c_N \vert} - 3 \frac{\overline{c_N}^3}{\vert c_N \vert^3}c_{3N} \right)} \\
* &  \displaystyle{\vert c_N \vert - \frac{1}{4} \Re\left( \frac{c_N c_{-N}}{\vert c_N \vert} + 3 \frac{\overline{c_N}^3}{\vert c_N \vert^3}c_{3N} \right)}
\end{array}\right).
$$
In addition
$$
D^2 \gamma_N(\lambda^* + \pi,0,0) = - D^2 \gamma_N(\lambda^*,0,0).
$$
If the determinant of $\mathcal{M}(p)$ does not vanish then Proposition \ref{exis} and Theorem \ref{st} apply.
In this way we are led to a very explicit result.

\begin{proposition}\label{expli}
Assume that $p(t)$ is such that $c_0 = c_{2N} = 0$ and $c_N \neq 0$. In addition 
$$
\textnormal{det}\,\mathcal{M}(p) \neq 0.
$$
Then, for small $\ep > 0$, the equation \eqref{2} has two smooth families of $2\pi$-periodic solutions satisfying
$$
x(t,\ep) = e^{i (\lambda^* + Nt)} + O(\ep), \qquad \tilde{x}(t,\ep) = - e^{i(\lambda^* + Nt)} + O(\ep)
$$ 
for $\lambda^*$ satisfying $e^{i\lambda^*} =\frac{c_N}{|c_N|}$. If $\textnormal{det}\,\mathcal{M}(p) > 0$
the first family is elliptic and the second is unstable in the Lyapunov sense; if $\textnormal{det}\,\mathcal{M}(p) < 0$ both families are unstable.
\end{proposition}

The example in the Introduction corresponds to $N=1$, $\lambda^* = 0$ and 
$$
\mathcal{M}(p) = \left( \begin{array}{cc}
\vspace{0.1cm}
1 + \frac{1}{4} \Re(a) & 
\frac{1}{4} \Im (a) \\
* &  1 -\frac{1}{4} \Re (a)
\end{array}\right).
$$

It is also possible to obtain some more or less explicit conditions in the non-circular case. The computations are more involved and we have been unable to find simple conditions like in Proposition \ref{expli}. 
It seems that the classical Fourier expansion solving the Kepler equation $u-e\sin u =t$ plays a role when $p(t)$ is a trigonometric polynomial.

\section{Appendix}

\begin{proof}[Proof of Proposition \ref{analytic}]
The main point in the proof is to show that the function $\mathcal{P}= \mathcal{P}(\lambda,\Lambda;\eta,\xi)$ is real analytic at each point
$(\lambda,\Lambda;0,0) \in \Omega \setminus \Omega^*$. If this is given for granted the proof follows easily. We already know that $\mathcal{P}$ is 
a symplectic diffeomorphism between $\Omega^*$ and $\mathcal{E}_+ \setminus \mathcal{C}$ and a homeomorphism from $\Omega$ onto $\mathcal{E}_+$.
Therefore it is enough to prove that $\mathcal{P}$ is a local symplectic diffeomorphism around each point on $\Omega \setminus \Omega^*$. To this end we invoke the identity \eqref{sym}. In principle it is valid on $\Omega^*$ but the continuity of the derivatives of $\mathcal{P}$ implies that it also holds on $\Omega$. In consequence $\textnormal{det}\,\mathcal{P}' = 1$ on $\Omega$ and the inverse function theorem can be applied at each point $(\lambda,\Lambda;0,0)$.

Let us now concentrate on the analiticity of the coordinates of $\mathcal{P}$. We start with some auxiliary functions. By direct computations the eccentricity can be expressed in terms of Poincar\'e coordinates,
$$
e = \sqrt{\eta^2 + \xi^2} \,\sqrt{\frac{1}{\Lambda}-\frac{\eta^2 + \xi^2}{4\Lambda^2}}.
$$
This implies that $e^2$ is analytic in $(\Lambda,\eta,\xi)$. The same can be said for 
$\gamma:= \tfrac{e}{\sqrt{\eta^2 + \xi^2}}$. Define $A:= e \cos h = \gamma \xi$,
$B:= e \sin h= \gamma\eta$. They are also analytic and 
$A(\Lambda,0,0)=B(\Lambda,0,0) = 0$. Define now $\alpha:=e \cos l$,
$\beta:=e\sin l$. Using $\lambda = l + g$ we can express them in terms of $A,B,\cos \lambda,\sin\lambda$. This implies that $\alpha$ and $\beta$ are analytic functions in the variables $(\lambda,\Lambda;\eta,\xi)$ and
$\alpha(\lambda,\Lambda;0,0) = \beta(\lambda,\Lambda;0,0) = 0$.

Kepler's equation $u - e \sin u = l$ can be rewritten as
$$
u-l = \alpha \sin(u-l) + \beta \cos (u-l)
$$
and the implicit function theorem can be applied to deduce that $u-l = e\sin u$ is analytic in
$(\lambda,\Lambda;\eta ,\xi )$. Also $e \cos u = \alpha \cos (u-l) - \beta \sin(u-l)$ is analytic.
From $u+g = u-l+\lambda$ we deduce that $\cos(u+g)$ and $\sin(u+g)$ are analytic. Finally we observe that $e^2 \cos (u-g)$ and  $e^2 \sin (u-g)$ are analytic. This follows from
$u-g = u+l-\lambda$ and
$$
e^2 \cos(u+l) = (\alpha^2-\beta^2)\cos(u-l) - 2\alpha\beta\sin(u-l),
$$
$$
e^2 \sin (u+l) = (\alpha^2-\beta^2)\sin(u-l) + 2\alpha\beta\cos (u-l).
$$
To complete the proof we express the coordinates of $\mathcal{P}$ in terms of the  previous functions. In particular, 
for $x_1$ and $y_1$,
$$
\frac{x_1}{\Lambda^2} = \frac{1+\sqrt{1-e^2}}{2}\cos(u+g) + \frac{1}{2(1+\sqrt{1-e^2})}e^2 \cos(u-g) - A
$$
and
$$
\Lambda y_1 = \frac{1}{1-e\cos u}\left[ - \frac{1}{2(1+\sqrt{1-e^2})}e^2\sin(u-g) - \frac{1+\sqrt{1-e^2}}{2}\sin(u+g)\right].
$$
\end{proof}

We shall now explain how to obtain the formulas presented at the end of Section \ref{s3}. Given $\mathcal{P} = (x,y)^t$ and
$x = x(\lambda,\Lambda,\eta,\xi)$, $x=(x_1,x_2)^t$, we deduce from the definition of the Poincar\'e coordinates that
$$
x(\lambda,\Lambda,0,0) = \Lambda^2 \left( \begin{array}{l}
\vspace{0.2cm}
\cos\lambda \\
\sin \lambda 
\end{array}\right).
$$
The derivatives $\frac{\partial x}{\partial \lambda}$ and $\frac{\partial^2 x}{\partial \lambda^2}$
can be computed from here. To compute the remaining derivatives we consider the change of variables
$r^2 = 2H$, $\eta = r \sin h$, $\xi = r \cos h$  and observe that the functions
$e = e(\Lambda,r)$ and $u=u(\lambda,\Lambda,r,h)$ are analytic at $r = 0$.
This is a consequence of the formulas
$$
e = r \sqrt{\frac{1}{\Lambda} - \frac{r^2}{4\Lambda^2}}, \qquad u - e \sin u = \lambda + h.
$$
In consequence the function
\begin{equation}\label{ft}
\tilde x(\lambda,\Lambda,r,h) = \Lambda^2 R[-h] \left( \begin{array}{l}
\vspace{0.2cm}
\cos u - e \\
\sqrt{1-e^2}\sin u 
\end{array}\right)
\end{equation}
is analytic at $r = 0$, $\Lambda >  0$, $\lambda,h \in \mathbb{T}$.

Since $x(\lambda,\Lambda,\eta,\xi)= \tilde x(\lambda,\Lambda,r,h)$ we can differentiate at each point with $r > 0$ to obtain
$$
\frac{\partial \tilde x}{\partial r} = \frac{\partial x}{\partial\eta}\sin h + \frac{\partial x}{\partial\xi} \cos h 
$$
and
$$
\frac{\partial^2 \tilde x}{\partial r^2} = \frac{\partial^2 x}{\partial\eta^2}\sin^2 h + 
2 \frac{\partial^2 x }{\partial \eta \partial \xi} \sin h \cos h +
\frac{\partial^2 x}{\partial\xi^2} \cos^2 h.
$$
Letting $r \to 0^+$ and selecting appropriate values of $h$ we express the derivatives of $x$ with respect to $\eta$ and $\xi$ in terms of
$\frac{\partial\tilde x}{\partial r}$ and $\frac{\partial^2\tilde x}{\partial r^2}$.
For instance, the choice $h = \frac{\pi}{2}$ leads to the identities
\begin{equation}\label{ide1}
\frac{\partial x}{\partial \eta}(\lambda,\Lambda,0,0) = \frac{\partial \tilde x}{\partial r}\left(\lambda,\Lambda,0,\frac{\pi}{2}\right)
\end{equation}
and
\begin{equation}\label{ide2}
\frac{\partial^2 x}{\partial \eta^2}(\lambda,\Lambda,0,0) = \frac{\partial^2 \tilde x}{\partial r^2}\left(\lambda,\Lambda,0,\frac{\pi}{2}\right).
\end{equation}
Similar identities for the derivatives with respect to $\xi$ are obtained with the choice $h = 0$. Finally the identity
\begin{equation}\label{ide3}
\frac{\partial^2 x}{\partial \eta\partial\xi}(\lambda,\Lambda,0,0) = \frac{\partial^2 \tilde x}{\partial r^2}\left(\lambda,\Lambda,0,\frac{\pi}{4}\right)
-\frac{1}{2}\left[\frac{\partial^2 x}{\partial \eta^2}(\lambda,\Lambda,0,0) + \frac{\partial^2 x}{\partial \xi^2}(\lambda,\Lambda,0,0)\right]
\end{equation}
is obtained for $h = \frac{\pi}{4}$.

Our next task is to compute $\frac{\partial\tilde x}{\partial r}$ and $\frac{\partial^2\tilde x}{\partial r^2}$ at $r = 0$.
Differentiating \eqref{ft} with respect to $r$ and evaluating at $r=0$ we obtain
\begin{equation}\label{fora}
\frac{\partial \tilde x}{\partial r}(\lambda,\Lambda,0,h) = \frac{1}{2}\Lambda^{3/2} R[-h] \left( \begin{array}{l}
\vspace{0.2cm}
-3 + \cos 2(\lambda+h) \\
\sin 2(\lambda+h) 
\end{array}\right),
\end{equation}
where we have employed
$$
e(\Lambda,0) = 0, \qquad \frac{\partial e}{\partial r}(\Lambda,0) = \sqrt{\frac{1}{\Lambda}},
$$
$$
u(\lambda,\Lambda,0,h) = \lambda + h, \qquad \frac{\partial u}{\partial r}(\lambda,\Lambda,0,h) = \sqrt{\frac{1}{\Lambda}} \sin(\lambda+h)
$$
together with the trigonometric identity $1 + \sin^2 \theta = \tfrac{3-\cos 2\theta}{2}$.

Combining \eqref{fora} with \eqref{ide1} we compute $\frac{\partial x}{\partial \eta}$
and $\frac{\partial x}{\partial \xi}$ is computed in a similar way. Differentiating with respect to $\lambda$ the identities
$$
\frac{\partial x}{\partial \eta}(\lambda,\Lambda,0,0) = \frac{\Lambda^{3/2}}{2}\left( 3i+ie^{2i\lambda}\right), \qquad
\frac{\partial x}{\partial \xi}(\lambda,\Lambda,0,0) = \frac{\Lambda^{3/2}}{2}\left( -3+e^{2i\lambda}\right)
$$
we compute $\frac{\partial^2 x}{\partial\lambda\partial\eta}$ and $\frac{\partial^2 x}{\partial\lambda\partial\xi}$. By now the only derivatives that have not been computed are $\frac{\partial^2 x}{\partial\eta^2}$, $\frac{\partial^2 x}{\partial\xi^2}$ and $\frac{\partial^2 x}{\partial\eta\partial\xi}$.
They will follow from \eqref{ide2}, \eqref{ide3} and the computation of $\frac{\partial^2 \tilde x}{\partial r^2}$. To this end we differentiate twice the identity \eqref{ft} with respect to $r$ and evaluate at $r=0$ to obtain
$$
\frac{\partial^2 \tilde x}{\partial r^2}(\lambda,\Lambda,0,h) = \Lambda R[-h] \left( \begin{array}{l}
\vspace{0.2cm}
-3 \sin^2(\lambda+h)\cos (\lambda+h) \\
-\sin (\lambda+h)-\sin^3(\lambda+h)+2\cos^2(\lambda+h)\sin(\lambda+h) 
\end{array}\right),
$$
where we are using
$$
\frac{\partial^2 e}{\partial r^2}(\Lambda,0) = 0 \quad \mbox{ and } \quad
\frac{\partial^2 u}{\partial r^2}(\lambda,\Lambda,0,h) = \frac{2\sin(\lambda+h)\cos(\lambda+h)}{\Lambda}.
$$
From the trigonometric identities
$$
\sin^2\theta \cos\theta = \frac{1}{4}(\cos\theta - \cos 3\theta), \quad
-\sin\theta-\sin^3\theta + 2 \cos^2\theta\sin\theta = \frac{-5\sin\theta + 3 \sin3\theta}{4}
$$
we express $\frac{\partial^2 \tilde x}{\partial r^2}$ in terms of 
$e^{i(\lambda+h)}$ and $e^{3i(\lambda+h)}$, precisely
$$
\frac{\partial^2 \tilde x}{\partial r^2}(\lambda,\Lambda,0,h) = \Lambda e^{-ih}
\left(-e^{i(\lambda+h)} + \frac{1}{4}\left( e^{-i(\lambda+h)} + 3 e^{3i(\lambda+h)}\right) \right).
$$
The derivatives $\frac{\partial^2 x}{\partial\eta^2}$ and $\frac{\partial^2 x}{\partial\xi^2}$ 
are computed evaluating $\frac{\partial^2 \tilde x}{\partial r^2}$ for $h=\tfrac{\pi}{2}$ and
$h= 0$. Finally we compute $\frac{\partial^2 x}{\partial\eta\partial\xi}$
by an evaluation of $\frac{\partial^2 \tilde x}{\partial r^2}$ at $h= \frac{\pi}{4}$
combined with \eqref{ide3}.

\subsection*{Acknowledgments}

The second author (R.O.) would like to express his thanks to Dr. Lei Zhao for introducing him to the Poincar\'e coordinates and also for discussions on some chapters of the always surprising book \cite{Poi92}.
\smallbreak
Work partially supported by the 
ERC Advanced Grant 2013 n. 339958
``Complex Patterns for Strongly Interacting Dynamical Systems - COMPAT'' and by the GNAMPA Project 2015 ``Equazioni Differenziali
Ordinarie sulla retta reale'' (A.B.) and by the project MTM2014-52232-P Spain (R.O.).

\vspace{1 cm}
\normalsize

Authors' addresses:
\bigbreak
\medbreak
\indent Alberto Boscaggin \\
\indent Dipartimento di Matematica, Universit\`a di Torino, \\
\indent Via Carlo Alberto 10, I-10123 Torino, Italy \\
\indent e-mail: alberto.boscaggin@unito.it \\

\medbreak
\indent Rafael Ortega\\
\indent Departamento de Matem\'atica Aplicada, Universidad de Granada,\\
\indent E-18071 Granada, Spain\\
\indent e-mail: rortega@ugr.es \\

\end{document}